\documentclass[11pt,reqno,letterpaper]{amsart}
\usepackage{color}
\usepackage[colorlinks=true, allcolors=blue,backref=page]{hyperref}
\usepackage{amsmath, amssymb, amsthm}
\usepackage{mathrsfs}
\usepackage{mathtools}
\usepackage[noabbrev,capitalize,nameinlink]{cleveref}
\crefname{equation}{}{}
\usepackage{fullpage}
\usepackage[noadjust]{cite}
\usepackage{graphics}
\usepackage{pifont}
\usepackage{tikz}
\usepackage{bbm}
\usepackage[T1]{fontenc}

\usetikzlibrary{arrows.meta}

\usepackage{environ}
\usepackage{framed}
\usepackage{url}
\usepackage[linesnumbered,ruled,vlined]{algorithm2e}
\usepackage[noend]{algpseudocode}
\usepackage[labelfont=bf]{caption}
\usepackage{cite}
\usepackage{framed}
\usepackage[framemethod=tikz]{mdframed}
\usepackage{appendix}
\usepackage{graphicx}
\usepackage[textsize=tiny]{todonotes}
\usepackage{tcolorbox}
\usepackage{enumerate}
\allowdisplaybreaks[1]
\usepackage{enumerate}
\usepackage{stmaryrd}
\usepackage[margin=1in]{geometry}

\usepackage[shortlabels]{enumitem}
\crefformat{enumi}{#2#1#3}
\crefrangeformat{enumi}{#3#1#4 to~#5#2#6}
\crefmultiformat{enumi}{#2#1#3}
{ and~#2#1#3}{, #2#1#3}{ and~#2#1#3}

\DeclareSymbolFont{symbolsC}{U}{pxsyc}{m}{n}
\SetSymbolFont{symbolsC}{bold}{U}{pxsyc}{bx}{n}
\DeclareFontSubstitution{U}{pxsyc}{m}{n}
\DeclareMathSymbol{\medcircle}{\mathbin}{symbolsC}{7}

\crefname{algocf}{Algorithm}{Algorithms}

\crefname{equation}{}{} 
\AtBeginEnvironment{appendices}{\crefalias{section}{appendix}} 

\usepackage[color,final]{showkeys} 

\colorlet{refkey}{orange!20}
\colorlet{labelkey}{blue!30}

\crefname{algocf}{Algorithm}{Algorithms}

\numberwithin{equation}{section}
\newtheorem{theorem}{Theorem}[section]
\newtheorem{proposition}[theorem]{Proposition}
\newtheorem{lemma}[theorem]{Lemma}

\crefname{claim}{Claim}{Claims}

\newtheorem{conjecture}[theorem]{Conjecture}
\newtheorem*{question*}{Question}

\theoremstyle{definition}
\newtheorem{definition}[theorem]{Definition}

\newtheorem*{definition*}{Definition}

\theoremstyle{remark}
\newtheorem*{remark}{Remark}

\newcommand{\norm}[1]{\bigg\lVert#1\bigg\rVert}

\newcommand{\mb}{\mathbb}

\newcommand{\mbm}{\mathbbm}
\newcommand{\mc}{\mathcal}

\newcommand{\mr}{\mathrm}

\newcommand{\on}{\operatorname}

\newcommand{\eps}{\varepsilon}

\let\originalleft\left
\let\originalright\right
\renewcommand{\left}{\mathopen{}\mathclose\bgroup\originalleft}
\renewcommand{\right}{\aftergroup\egroup\originalright}

\allowdisplaybreaks

\newcommand{\ignore}[1]{}

\title{The random graph process is globally synchronizing}
\author[A1]{Vishesh Jain}
\address{Department of Mathematics, Statistics, and Computer Science, University of Illinois Chicago, Chicago, IL, 60607 USA}
\email{visheshj@uic.edu}

\author[A2]{Clayton Mizgerd}
\address{Department of Mathematics, Statistics, and Computer Science, University of Illinois Chicago, Chicago, IL, 60607 USA}
\email{cmizge2@uic.edu}

\author[A3]{Mehtaab Sawhney}
\address{Department of Mathematics, Columbia University, New York, NY 10027}
\email{m.sawhney@columbia.edu}

\newcommand{\paren}[1]{\left(#1\right)}

\begin{document}

\begin{abstract}
The homogeneous Kuramoto model on a graph $G = (V,E)$ is a network of $|V|$ identical oscillators, one at each vertex, where every oscillator is coupled bidirectionally (with unit strength) to its neighbors in the graph. A graph $G$ is said to be globally synchronizing if, for almost every initial condition, the homogeneous Kuramoto model converges to the all-in-phase synchronous state. Confirming a conjecture of Abdalla, Bandeira, Kassabov, Souza, Strogatz, and Townsend, we show that with high probability, the random graph process becomes globally synchronizing as soon as it is connected. This is best possible, since connectivity is a necessary condition for global synchronization. 
\end{abstract}

\maketitle

\section{Introduction}
\paragraph{\bf The Kuramoto model and global synchronization}
The \emph{Kuramoto model}, first proposed in the physics literature by Yoshiki Kuramoto \cite{kuramoto1975self}, is a well-studied mathematical model for the behavior of multiple coupled oscillators. Formally, given a graph $G = (V,E)$ with adjacency matrix $A \in \mb{R}^{V\times V}$, we associate to each vertex $v \in V$ an oscillator with intrinsic frequency $\omega_{v}$ and denote the $\mb{S}^1$-valued phase of the oscillator at time $t$ by $\theta_v(t)$. The time evolution of the phases of the oscillators is governed by the following system of differential equations:
\[\frac{d\theta_v}{dt} = \omega_{v} - \sum_{u \in V}A_{u,v}\sin(\theta_{v} - \theta_{u}), \qquad \forall v \in V,\]
In this work, we will be concerned only with the \emph{homogeneous} case where all oscillators have the same intrinsic frequency, i.e.~$\omega_v = \omega$ for all $v \in V$. In the homogeneous case, by working in the rotating frame $\theta(t) \leftarrow \theta(t) - \omega(t)$, we can recast the preceding system of differential equations as the gradient flow
\[\frac{d\theta}{dt} = -\nabla \mc{E}_G(\theta),\]
where the energy function $\mc{E}_G : (\mb{S}^1)^{V} \to \mb{R}$ is given by
\[\mc{E}_G(\theta) = \frac{1}{2}\sum_{u,v \in V}A_{u,v}(1-\cos(\theta_u - \theta_v)).\]
The energy function is always non-negative and is minimized when all oscillators have the same phase, i.e.~$\theta_u = \theta_v$ for all $u,v \in V$. Such states are called \emph{fully synchronized states}. 

We say that a graph $G = (V,E)$ is \emph{globally synchronizing} if the homogeneous Kuramoto model on $G$, initialized at a uniformly random state $\theta \in (\mb{S}^1)^{V}$, converges to a fully synchronized state with probability one (over the randomness of the initial state). It is well-known (see,~e.g.~\cite{abdalla2022expander}) that in order to show that $G$ is globally synchronizing, it suffices to show that the energy function $\mc{E}_G$ has no spurious local minima,~i.e.~the only states with $\nabla \mc{E}_G (\theta) = 0$ and $\nabla^2 \mc{E}_G(\theta)$ positive semidefinite are the fully synchronized states. Note that a trivial necessary condition for a graph to be globally synchronizing is that it is connected; else having the oscillators only agree within connected components is also a global minimum of the energy function. 

\paragraph{\bf The Kuramoto model on random graphs} In recent years, much work has been devoted to studying the Kuramoto model on (sparse) random graphs, as a proxy for (sparse) realistic networks. Recall that for $p \in (0,1)$, the (binomial) Erd\H{o}s-R\'enyi model $G(n,p)$ is the distribution on graphs with $n$ vertices where each edge is present independently with probability $p$. Ling, Xu, and Bandeira \cite{ling2019landscape} showed that with high probability (i.e.~with probability tending to $1$ as $n$ tends to infinity), $G \sim G(n,p)$ is globally synchronizing if $p = \Omega(n^{-1/3}\log{n})$ and conjectured that this can be improved to $\Omega(n^{-1}\log{n})$, which would be best possible up to the implicit constant. Subsequently, Kassabov, Strogatz, and Townsend \cite{kassabov2021sufficiently} obtained global synchronization for $p = \Omega(n^{-1}(\log n)^2)$. Recently, Abdalla, Bandeira, Kassabov, Souza, Strogatz, and Townsend \cite{abdalla2022expander} proved the conjecture of Ling, Xu, and Bandeira in a strong form: for any fixed $\varepsilon > 0$ and $p = (1+\varepsilon)n^{-1}\log{n}$, $G \sim G(n,p)$ is globally synchronizing with high probability. The constant $1$ is best possible in the sense that for $p < n^{-1}\log{n}$, $G \sim G(n,p)$ is no longer  connected with high probability (and therefore, not globally synchronizing). In fact, the work in \cite{abdalla2022expander} obtains this result from the deterministic result that all graphs with ``sufficiently good expansion'' are globally synchronizing; we discuss this in more detail later in this introduction. For instance, the result of \cite{abdalla2022expander} also implies that for $d \geq 600$, a random $d$-regular graph is globally synchronizing with high probability.
\\

\paragraph{\bf Connectivity and synchrony} Since $p = n^{-1}\log{n}$ is the ``sharp threshold'' for the connectivity of an Erd\H{o}s-R\'enyi random graph, Abdalla, Bandeira, Kassabov, Souza, Strogatz, and Townsend \cite[Section~7]{abdalla2022expander} conjectured that for random graphs, with high probability, the only obstruction to synchrony is connectivity. Formally, consider the random graph process $\{G(n,m)\}_{m \geq 0}$ which is a graph-valued stochastic process initialized with $G(n,0)$ being the empty graph on $n$ vertices and where $G(n,m+1)$ is obtained from $G(n,m)$ by adding a uniformly chosen missing edge. Let $\tau$ denote the (random) first time when the random graph process is connected; formally 
\[\tau = \min\{m: G(n,m)\text{ is connected}\}.\] 
The authors of \cite{abdalla2022expander} proposed the following conjecture to ``completely seal the gap between connectivity and synchrony''. 
\begin{conjecture}[\cite{abdalla2022expander}]
With high probability, $G(n,m)$ is globally synchronizing for all $m \geq \tau$.     
\end{conjecture}
\begin{remark}
There are two natural interpretations of this conjecture: the weaker interpretation is that for any time $m \geq \tau$, $G(n,m)$ is globally synchronizing with high probability. The stronger interpretation is that \emph{all} graphs $G(n,m)$ for $m\geq \tau$ are simultaneously globally synchronizing with high probability. The latter statement is \emph{not} a trivial consequence of the former as global synchronization is not a monotone property. For instance, every tree is globally synchronizing whereas a cycle of length greater than $5$ is not.
\end{remark}

Our main result is a confirmation of the stronger form of this conjecture. 

\begin{theorem}\label{thm:Gnm-synchronizing}
With notation as above, the following holds with high probability: for all $m\geq \tau$ simultaneously, $G(n,m)$ is globally synchronizing.    
\end{theorem}

\paragraph{\bf Defective expanders are globally synchronizing} As mentioned above, Abdalla et.~al \cite{abdalla2022expander} obtain their statement about random graphs as a consequence of a deterministic statement about expanders. We now define various terms to formally state their result.

For a graph $G = (V,E)$, we denote the adjacency matrix by $A_G$. For a matrix $M$, $\|M\|$ will throughout denote the spectral norm (i.e.~$\ell_2 \to \ell_2$ operator norm). Moreover, we reserve $J$ for the all-ones matrix. 

\begin{definition}\label{def:expander-1}
A graph $G = (V,E)$ is an $(n,d,\alpha)$-expander if $|V| = n$ and
\[ \norm{A_G - \frac dn J} \leq \alpha d. \]
\end{definition}

The main result in \cite{abdalla2022expander} also requires a slightly more bespoke notion of expander which accounts for varying degrees in a graph in a more careful manner. 

\begin{definition}\label{def:expander-2}
A graph $G = (V,E)$ is an is an $(n,d,\alpha,c^-,c^+)$-expander if $G$ is an $(n,d,\alpha)$-expander and in addition
\[ c^-d I \preceq D - A_G - dI + \frac dn J \preceq c^+dI. \]
 Here $M \preceq M'$ if and only if $M'-M$ is positive semidefinite.
\end{definition}

Observe that if $G$ is a $(n,d,\alpha,c^-,c^+)$-expander then the minimum degree of $G$ is at least $(1+c^{-})d$.

The main technical result of \cite{abdalla2022expander} is the following. 
\begin{theorem}[{\cite[Theorem~1.10]{abdalla2022expander}}]\label{thm:tech}
Suppose that $G$ is a $(n,d,\alpha,c^-,c^+)$-expander graph with $c^{-}>-1$, $\alpha\le 1/5$, and 
\[\max\Big\{\frac{64\alpha(1+2c^{+}-c^{-})}{(1+c^{-})^2}, \frac{64\alpha(1+c^{+})\log\big(\frac{1+c^{+}+\alpha}{2\alpha}\big)}{(1+c^{-})(1+5c^{+}-4c^{-})}\Big\}<1.\]
Then $G$ is globally synchronizing. 
\end{theorem}

Standard results in random graph theory show that for any $\varepsilon > 0$, with high probability, $G \sim G(n, (1+\varepsilon)n^{-1}\log{n})$ is such an expander, in large part because for such a $G$, the ratio of the maximum degree and the minimum degree is bounded by a constant depending only on $\varepsilon$. 

However, for the random graph process at connectivity, i.e.~$G(n,\tau)$, with high probability there is a vertex of degree $1$. By the remark following \cref{def:expander-2}, this forces $c^{-}$ very close to $-1$ and one can easily check that therefore, \cref{thm:tech} does not apply in this setting. 

Our main contribution is the introduction of the notion of a ``defected expander'', which is strong enough to guarantee global synchronization, but at the same time, is weak enough to hold with high probability for $G(n,\tau)$.

\begin{theorem}\label{thm:main}
Fix $\eps\in (0,1)$, $\alpha\in (0,1/5)$ and $d$ such that $\alpha d \geq 1$.

Let $G = (V,E)$ be a graph with $V = W \cup B$ and no isolated vertices. Suppose that $G$ is a $(n,d,\alpha)$-expander. Furthermore, let the maximum degree of $G[W]$ be $d_\mr{max}$ and the minimum degree be lower bounded by $\ell:= 2(\eps + \alpha)d$. Define 
\[\delta = \eps d/(d_\mr{max} -  2\varepsilon d).\]

Suppose that the following hold:
\begin{align}
\frac{8(1+\delta)(d_\mr{max} + 1)}{\delta^2 \exp(\frac{\eps}{40\alpha} \log(1+\delta))} & \leq 1, \label{eq:condition-for-k*}\\ 
\frac{8}{d_\mr{max} - 4\eps d} & \leq 1. \label{eq:condition-for-d}
\end{align}
Furthermore suppose that $B$ satisfies
\begin{enumerate}[label={(B\arabic*)}, leftmargin = \leftmargin + 1\parindent]
\item\label{item:B-small} $|B| \leq \alpha n$,
\item\label{item:B-empty} No edge of $G$ has both of its endpoints in $B$,
\item\label{item:B-sparse} For all $v \in B$, $\deg v \leq d_\mr{max} + 1$,
\item\label{item:no-2-neighbors} For all $v \in V$, $|N(v) \cap B| \leq 1$.
\end{enumerate}
Then $G$ is globally synchronizing.
\end{theorem}

\subsection{Notation}\label{subsec:not}
We use standard asymptotic notation throughout, as follows. For functions $f=f(n)$ and $g=g(n)$, we write $f=O(g)$ or $f \lesssim g$ to mean that there is a constant $C$ such that $|f(n)|\le C|g(n)|$ for sufficiently large $n$. Similarly, we write $f=\Omega(g)$ or $f \gtrsim g$ to mean that there is a constant $c>0$ such that $f(n)\ge c|g(n)|$ for sufficiently large $n$. Finally, we write $f\asymp g$ or $f=\Theta(g)$ to mean that $f\lesssim g$ and $g\lesssim f$, and we write $f=o(g)$ or $g=\omega(f)$ to mean that $f(n)/g(n)\to0$ as $n\to\infty$. Subscripts on asymptotic notation indicate quantities that should be treated as constants.

\subsection{Acknowledgements}  VJ is supported by NSF CAREER award DMS-2237646. CM is supported by NSF award ECCS-2217023 through the IDEAL institute. This research was conducted during the period MS served as a Clay Research Fellow. 

\subsection{Organization} In \cref{sec:proof-main}, we prove \cref{thm:main}. In \cref{sec:hitting-time}, we deduce \cref{thm:Gnm-synchronizing} from \cref{thm:main}. Since the proofs are quite short, we defer a discussion of the proof ideas to the respective sections.

\section{Proof of \texorpdfstring{\cref{thm:main}}{Theorem}}\label{sec:proof-main}

\subsection{Proof strategy}\label{sec:mixing-implies-synchronization}

Recall that we wish to show that the Kuramoto model on $G$ has no spurious local minima. For consistency with the literature, we refer to local minima as stable states.   

\begin{definition}
        We say that $\theta \in (\mb{S}^1)^V$ is a stable state on a graph $H = (V,E)$ if it is a local minimum of $\mc{E}_H$, i.e.~$\nabla \mc{E}_H(\theta) = 0$ and $\nabla^2 \mc{E}_H(\theta) \succeq 0$.
\end{definition}

    Note that every fully synchronized state, i.e.~$\theta_u = \theta_v$ for all $u,v \in V$ is a stable state. By a non-trivial stable state, we mean a stable state which is not fully synchronized. 

    As in \cite{abdalla2022expander}, we proceed by contradiction; suppose there exists a non-trivial stable state $\theta \in (\mb{S}^1)^V$. By performing a global phase change if necessary, we may (and will always) assume without loss of generality that $\rho_1(\theta) := |V|^{-1}\sum_{v \in V}e^{i\theta_v} \in [0,1]$. Viewing $\mb{S}^{1}$ as $(-\pi, \pi]$, for $\beta \in \mb{S}^{1}$, we define 
        \[ \mc{C}_\beta = \mc{C}_{\beta}(\theta) := \{ v \in V : |\theta_v| \geq \beta \}. \]
Observe that $|\mc{C}_\beta|\ge |\mc{C}_{\beta'}|$ if $\beta\le \beta'$.

The well-known half-circle lemma asserts that the phases of a non-trivial stable state cannot lie in a common half-circle (see e.g.~\cite[Lemma 2.8]{abdalla2022expander}).

\begin{lemma}\label{lem:half-circle}
Suppose $\theta$ is a nontrivial stable state on a connected graph $H$ with $\rho_1(\theta) \in [0,1]$.  Then $\mc{C}_{\pi/2} \ne \emptyset$.
\end{lemma}

Starting from \cref{lem:half-circle}, we will iteratively construct a sequence of angles to arrive at an angle $\beta$ for which satisfies
    \begin{equation}
    \label{eq:contradiction}
    |\mc{C}_\beta| \sin^2(\beta) > \frac{5 \alpha^2 n}{2}.
    \end{equation}
However, this contradicts the following result due to Abdalla et.~al. \cite{abdalla2022expander}.
\begin{lemma}[{\cite[eq.~(4.2)]{abdalla2022expander}}]\label{lem:contradiction}
Let $H$ be an $(n,d,\alpha)$-expander for $\alpha \leq 1/5$. For every stable state $\theta$ with $\rho_1(\theta) \in [0,1]$ and $\beta \in \mb{S}^1$, we have
        \[ |\mc{C}_\beta| \sin^2(\beta) \leq \frac{5 \alpha^2 n}{2}. \]
\end{lemma}

The key technical challenge in our work is that $G$ in \cref{thm:main} is only a defective expander. In particular the expansion properties available to us when dealing with the exceptional low degree vertices are substantially different than those implied by \cref{def:expander-2}, which are used crucially in \cite{abdalla2022expander}. We discuss our main ingredient to overcome this in \cref{sec:amplification}. 

\subsection{Expansion properties of \texorpdfstring{$G$}{G}}\label{sec:expansion-on-G}

We first prove that the core $G[W]$ satisfies essentially the same spectral condition as $G$. 

\begin{lemma}\label{lem:spectral-GW}
Let $G$ be as in \cref{thm:main}. Then $G[W]$ is a $(|W|,d,2\alpha)$-expander.
\end{lemma}

\begin{proof}
    Let $n' = |W|$. We have
    \begin{align*} \norm{A_{G[W]} - \frac{d}{n'}J_{n'}} & \leq \norm{A_{G[W]} - \frac dn J_{n'}} + \norm{\paren{\frac{d}{n'} - \frac dn} J_{n'}} \\ & \leq \norm{A_G - \frac dn J_n} + d\paren{\frac{n - n'}{n}} \leq \alpha d + \alpha d. \end{align*}
Here we have used the triangle inequality, that $\|J_{n'}\| = n'$, and that the spectral norm of a submatrix is at most that of the original matrix. The final bound uses that $G$ is an $(n,d,\alpha)$-expander and \cref{item:B-small}.
\end{proof}

The spectral expansion of the core in the above sense, along with the assumed degree bounds on $G[W]$, imply strong combinatorial expansion properties on $G[W]$. Our proofs employ the next two lemmas, which are specializations of corresponding results in \cite{abdalla2022expander}. Throughout, for a graph $H = (V,E)$ with adjacency matrix $A_H$ and $X,Y \subseteq V$, we define 
\[e(X,Y) = e_H(X,Y) := \sum_{x \in X, y \in Y}(A_H)_{x,y}.\]
    
\begin{lemma}[{\cite[Lemma 3.3]{abdalla2022expander}}]
Let $H$ be an $(n,d,\alpha)$-expander.  Then for $X \subseteq V(H)$, we have
\begin{align}
e(X,X) & \leq \frac dn |X|^2 + \alpha d |X|. \label{eq:XX-bound}
\end{align}
\end{lemma}

\begin{lemma}[{\cite[Proposition 3.2 and Lemma 3.7]{abdalla2022expander}}]\label{lem:XY-bound-original}
Fix $\eps>0$. Let $H$ be an $(n,d,\alpha)$-expander with minimum degree at least $2(\eps + \alpha)d$ and maximum degree at most $d_\mr{max}$. 

If $X \subseteq Y \subset V(H)$ satisfy $|Y| \leq n/2$ and $|Y| \leq (1+\delta)|X|$ for $\delta = \eps d/(d_\mr{max} - 2\eps d)$, then
\[ \eps \frac dn |X| |Y^c| \leq e(X,Y^c). \]
\end{lemma}
\begin{proof}
This is an application of \cite[Proposition 3.2]{abdalla2022expander} with $c^+ = (d_\mr{max}/d) - 1 + \alpha$, $c^- = 2\eps -1 + \alpha$, followed by an application of \cite[Lemma 3.7]{abdalla2022expander}. We remark that one may obtain a stronger lower bound of $(\eps+\alpha) d/n|X||Y^c|$, but the first term will be sufficient in our applications. 
\end{proof}

We will apply these lemmas for $G[W]$ together with our assumptions on the defects $B$ to prove nearly the same properties for $G$.

\begin{lemma}\label{lem:XV-bound}
Let $G$ be as in \cref{thm:main}. For $X \subseteq V$, we have
\[ e(X,V) \leq (d_\mr{max} + 1) |X|. \]
\end{lemma}
\begin{proof}
Note that
\begin{align*}
e(X,V) & = e(X \cap W, W) + e(X \cap W, B) + e(X \cap B, V) \\
& \leq d_\mr{max}|X \cap W| + |X \cap W| + (d_\mr{max} + 1) |X \cap B| \\
& = (d_\mr{max} + 1)|X|.
\end{align*}
Here we bound the first term using the definition of $d_\mr{max}$, the second term using \cref{item:no-2-neighbors}, and the last term using \cref{item:B-sparse}.   
\end{proof}

\begin{lemma}\label{lem:XYc-to-XX}
Let $G$ be as in \cref{thm:main} and recall $\delta = (\eps d)/(d_\mr{max} - 2\varepsilon d)$.

For $X \subseteq Y \subseteq V$ with $|Y| \leq n/2$, $|Y| \leq (1+\delta)|X|$ and $|X| \leq \alpha n$, we have
        \[ e(X,Y^c) > \frac{\eps}{20\alpha} e(X,X). \]
    \end{lemma}

\begin{proof}
For $S\subseteq V$, let $S_W := S \cap W$ and $S_B := S \cap B$.  Expand
\[ e(X,X) = e(X_W, X_W) + 2e(X_W, X_B) + e(X_B,X_B) \leq e(X_W, X_W) + 2|X_W|; \]
the inequalities used follow from \cref{item:no-2-neighbors} and \cref{item:B-empty}. Combining this with \cref{eq:XX-bound} (recalling $G[W]$ is a $(|W|,d,2\alpha)$-expander), we have
        \begin{align*}
            e(X,X) & \leq \frac d{|W|} |X_W|^2 + 2\alpha d |X_W| + 2|X_W| \leq \frac{21}{4}\alpha d|X_W|.
        \end{align*}
In the last inequality, we use $|X_W|/|W| \leq |X|/(n-|B|) \leq \alpha/(1- \alpha) \leq 5\alpha/4$ and $1 \leq \alpha d$.  Now, using the lower bound from \cref{lem:XY-bound-original},
        \begin{align*}
            e(X,Y^c) & \geq e(X_W, (Y^c)_W) \geq \eps \frac d{|W|} |X_W| |(Y^c)_W| \geq \frac{4\eps}{21\alpha} \cdot \frac{|(Y^c)_W|}{|W|} \cdot e(X,X).
        \end{align*}
        Finally, note that $|(Y^c)_W|/|W| \geq (|Y^c| - |B|)/n \geq \frac12 - \alpha \geq \frac3{10}$, and that $12/210 > 1/20$.
    \end{proof}

\subsection {Amplification tools}
\label{sec:amplification}

We will construct a sequence of angles controlling both $\beta$ and $|\mc{C}_\beta|$ until we can contradict \cref{lem:contradiction}.  Our key amplification tools are \cref{lem:ratio-small,lem:ratio-large}, which allow us to inductively grow the size of $\mc{C}_\beta$ without decreasing $\beta$ too much. While \cref{lem:ratio-small} is similar to \cite[Lemma~4.3]{abdalla2022expander}, there are several key differences in both the statement and the proof of \cref{lem:ratio-large}, as compared to the corresponding \cite[Lemma~4.4]{abdalla2022expander}. Notably, the strong expansion conditions in \cite{abdalla2022expander} guarantee that for any $X, Y\subseteq V$ with $Y$ sufficiently large, $e(X, Y) > \eps d/n |X||Y|$; this is crucial in the proof of \cite[Lemma~4.4]{abdalla2022expander}. A global bound of this nature is no longer true in our case, for instance, if $X$ happens to consist only of vertices of constant degree. We therefore develop a version which relies only on weaker expansion properties of $G$ (as in \cref{sec:expansion-on-G}), along with the connectivity of $G$ (see, in particular, Case 2 in the proof of \cref{lem:ratio-large}); the price we pay is that we need the gap between $\beta_k - \beta_{k+1}$ to be much larger (logarithmic in $|V|$, for our main application) than in \cite{abdalla2022expander} in order to guarantee geometric growth in the size of $|\mc{C}_{\beta_i}|$. Nevertheless, we show that this is sufficient for the overall amplification argument, by iterating \cref{lem:ratio-small} several further times so that $|\mc{C}_{\pi/2}|/|\mc{C}_\beta|$ is small enough to recoup this additional (logarithmic) error.

The proofs of the main results of this subsection need the following lemma.

\begin{lemma}\label{lem:kernel-stability}
        Let $\theta$ be a stable state with $\rho_1(\theta) \in [0,1]$.  For all $0 < \gamma < \beta \leq \pi/2$,
        \begin{equation*} e(\mc{C}_\beta,\mc{C}_\beta) \geq e(\mc{C}_{\pi/2}, \mc{C}_{\beta}) \geq \sin(\beta - \gamma) e(\mc{C}_\beta, \mc{C}_\gamma^c). \end{equation*}
    \end{lemma}
    \begin{proof}
        The first inequality follows since $\mc{C}_{\pi/2} \subseteq \mc{C}_{\beta}$. The second inequality is implicit in the proof of \cite[Lemma~4.4]{abdalla2022expander} and \cite[eq.~(4.3)]{abdalla2022expander}; we sketch the details for completeness. Define $K:(-\pi, \pi] \times (-\pi, \pi] \to \mb{R}$ by $K(\alpha, \beta) = \sin(|\alpha| - \min\{|\beta|,\pi/2\})$. Direct computation shows that
        \[K(\alpha,\beta) + K(\beta,\alpha) \leq \mbm{1}_{|x|\geq \pi/2}(\alpha) + \mbm{1}_{|x|\geq \pi/2}(\beta).\]
        Therefore, for $A$ the adjacency matrix,
        \begin{align*}
            \sum_{u,v \in \mc{C}_{\beta}}A_{u,v}K(\theta_u,\theta_v) &= \frac{1}{2}\sum_{u,v \in \mc{C}_{\beta}}A_{u,v}\left(K(\theta_u,\theta_v)+ K(\theta_u,\theta_v)\right)\\
            &\leq \frac{1}{2}\sum_{u,v \in \mc{C}_{\beta}}A_{u,v}\left(\mbm{1}_{|x|\geq \pi/2}(\theta_u) + \mbm{1}_{|x|\geq \pi/2}(\theta_v)\right) = e(\mc{C}_{\pi/2}, \mc{C}_{\beta}). 
        \end{align*}
        Hence, it suffices to show that $\sum_{u,v \in \mc{C}_{\beta}}A_{u,v}K(\theta_u, \theta_v) \geq \sin(\beta-\gamma)e(\mc{C}_{\beta}, C_{\gamma}^c)$. 

       From \cite[eq~(4.3)]{abdalla2022expander}, we have that
       \[\sum_{u,v\in \mc{C}_{\beta}}A_{u,v}K(\theta_u, \theta_v) + \sum_{u\in C_{\gamma}\setminus C_{\beta}, v\in \mc{C}_{\beta}}A_{u,v}K(\theta_u, \theta_v) + \sum_{u \in \mc{C}_{\gamma}^c, v\in \mc{C}_{\beta}}A_{u,v}K(\theta_u, \theta_v) \geq 0\]
       Moreover, direct computation shows that for $v \in \mc{C}_{\beta}$,
       \[
       K(\theta_u,\theta_v)\leq \begin{cases}
           1, \quad u \in \mc{C}_{\beta},\\
           0, \quad u \in \mc{C}_{\gamma}\setminus \mc{C}_\beta,\\
           \sin(\gamma-\beta), \quad u \in \mc{C}_\gamma^c.
       \end{cases}
       \]
       Therefore,
       \begin{align*}
           0 &\leq \sum_{u,v\in \mc{C}_{\beta}}A_{u,v}K(\theta_u, \theta_v) + \sum_{u\in \mc{C}_\gamma^c, v \in \mc{C}_{\beta}}A_{u,v}\cdot \sin(\gamma-\beta) \\
           &\leq  \sum_{u,v\in \mc{C}_{\beta}}A_{u,v}K(\theta_u, \theta_v) + e(C_{\beta}, C_{\gamma}^c)\cdot \sin(\gamma-\beta),
       \end{align*}
       so that $\sum_{u,v \in \mc{C}_{\beta}}A_{u,v}K(\theta_u, \theta_v) \geq e(C_{\beta}, C_{\gamma}^c)\cdot \sin(\beta-\gamma)$, as desired. \qedhere
       
    \end{proof}

We now state and prove our main amplification results. 

\begin{lemma}\label{lem:ratio-small}
Let $G$ be as in \cref{thm:main} and recall that $\delta = (\eps d)/(d_\mr{max} - 2 \varepsilon d)$. Let $\theta$ be a stable state on $G$ with $\rho_1(\theta) \in [0,1]$.

If the angles $0 < \gamma < \beta \leq \pi/2$ satisfy
        \[ \sin(\beta - \gamma) \geq \frac{20\alpha}{\eps}, \]
        then
        \[ |\mc{C}_\gamma| > \min\left\{ (1+\delta) |\mc{C}_\beta|, \alpha n \right\}. \]
    \end{lemma}

    \begin{proof}
     Suppose for contradiction that $|\mc{C}_\gamma| \leq (1+\delta) |\mc{C}_\beta|$ and $|\mc{C}_\gamma| \leq \alpha n$ (and so $|\mc{C}_\beta| \leq \alpha n$).  Then, we have
        \[ e(\mc{C}_\beta,\mc{C}_\beta) \geq \sin(\beta - \gamma) e(\mc{C}_\beta,\mc{C}_\gamma^c) > \sin(\beta - \gamma) \frac{\eps}{20\alpha} e(\mc{C}_\beta,\mc{C}_\beta). \]
        This contradicts our assumption on $\sin(\beta - \gamma)$. The first inequality here is by \cref{lem:kernel-stability} and the second is by \cref{lem:XYc-to-XX}.
    \end{proof}

\begin{lemma}\label{lem:ratio-large}
Let $G$ be as in \cref{thm:main} and recall that $\delta = (\eps d)/(d_\mr{max} - 2 \varepsilon d)$. Let $\theta$ be a stable state on $G$ with $\rho_1(\theta) \in [0,1]$.

If the angles $0 < \gamma < \beta \leq \pi/2$ satisfy
\begin{equation}\label{eq:sin-max-2} \sin(\beta-\gamma) \geq \frac{d_\mr{max} + 1}{\delta} \cdot \frac{|\mc{C}_{\pi/2}|}{|\mc{C}_\beta|} \end{equation}
then
\[ |\mc{C}_\gamma| > \min\left\{ \left( 1 + \delta \right) |\mc{C}_\beta|, n/2 \right\}. \]
\end{lemma}

    \begin{proof}
        By \cref{lem:XV-bound}, we have
        \begin{equation}\label{eq:C_beta-case-2} e(\mc{C}_{\pi/2},\mc{C}_\beta) \leq e(\mc{C}_{\pi/2},V) \leq (d_\mr{max} + 1)|\mc{C}_{\pi/2}|. \end{equation}
        Moreover, by \cref{lem:kernel-stability}, we have $e(\mc{C}_{\pi/2}, C_\beta) \geq \sin(\beta - \gamma)e(\mc{C}_{\beta}, \mc{C}_{\gamma}^c)$.  Suppose for contradiction that $|\mc{C}_\gamma| \leq (1+\delta)|\mc{C}_\beta|$ and $|\mc{C}_\gamma| \leq n/2$. We will lower bound $e(\mc{C}_{\beta}, \mc{C}_{\gamma}^c)$ to derive a contradiction to \cref{eq:sin-max-2}. There are two cases.

        \textbf{Case 1:} $|\mc{C}_\beta \cap B| \leq (\frac12 + \delta)|\mc{C}_{\beta}|$.  Then, by \cref{lem:XY-bound-original},
        \[ e(\mc{C}_\beta,\mc{C}_\gamma^c) \geq e(\mc{C}_\beta \cap W, \mc{C}_\gamma^c \cap W) \geq \eps \frac dn |\mc{C}_\beta \cap W| |\mc{C}_\gamma^c \cap W| > \frac\eps4 d \paren{\frac12 - \delta} |C_\beta|,  \]
        using $|\mc{C}^c_\gamma \cap W| \geq |\mc{C}^c_\gamma| - |B| \geq (\frac12 - \alpha)n > \frac14n$ and $|\mc{C}_\beta \cap W| \geq (\frac12 - \delta)|\mc{C}_\beta|$.  By \cref{eq:C_beta-case-2} and \cref{lem:kernel-stability}, we have
        \[ \sin(\beta - \gamma) < \frac{4(d_\mr{max} + 1)}{(\frac12 - \delta) \eps d} \cdot \frac{|\mc{C}_{\pi/2}|}{|\mc{C}_\beta|} = \frac{d_\mr{max} + 1}{\delta} \cdot \frac{|\mc{C}_{\pi/2}|}{|\mc{C}_\beta|} \cdot \frac{8}{d_\mr{max} -  4\eps d}. \]

        The last step comes from recalling the definition of $\delta$ and algebraic manipulations.  By \cref{eq:condition-for-d}, we have $8/(d_{\mr{max}} - 4\eps d)\le 1$, so this contradicts \cref{eq:sin-max-2}.

        \textbf{Case 2:} $|\mc{C}_\beta \cap B| > (\frac12 + \delta)|\mc{C}_\beta|$. Let $S = \mc{C}_\beta \cap B$.  Since $S\subseteq \mc{C}_\beta \subseteq \mc{C}_\gamma$ and $|\mc{C}_\gamma| \leq (1+\delta)|\mc{C}_\beta|$ by assumption, it follows that $|\mc{C}_\gamma \setminus S| = |\mc{C}_\gamma| - |S| \leq |C_\beta|/2$.  Moreover, $G$ has minimum degree at least 1, and since $S \subset B$, no two vertices in $S$ have an edge between them or any common neighbors by \cref{item:B-empty} and \cref{item:no-2-neighbors}.  Thus
        \[ e(\mc{C}_\beta,\mc{C}_\gamma^c) \geq e(S,\mc{C}_\gamma^c) = e(S,V) - e(S, \mc{C}_\gamma) = e(S,V) - e(S, \mc{C}_{\gamma}\setminus S) \geq |S| -  |\mc{C}_\gamma \setminus S| > \delta |\mc{C}_\beta|. \]
 Putting this together with \cref{eq:C_beta-case-2} and \cref{lem:kernel-stability} gives
        \[ \sin(\beta - \gamma) < \frac{d_\mr{max} + 1}{\delta} \cdot \frac{|\mc{C}_{\pi/2}|}{|\mc{C}_\beta|}, \]
        contradicting \cref{eq:sin-max-2}.
    \end{proof}

\subsection{Global synchronization for \texorpdfstring{$G$}{G}}  We are now prepared to prove our main technical result.

    \begin{proof}[Proof of \cref{thm:main}]

        Let $k^* := \lfloor \eps/(40\alpha) \rfloor$ and $\beta_0 = \pi/2$.  For integers $0 < k \leq k^*$, let
        \[ \beta_k = \beta_{k-1} - \sin^{-1}\paren{\frac{20\alpha}{\eps}}. \]

       Note that if $k^* \geq 1$, then $40\alpha \leq \eps$, so this is well-defined. 
        Let $\beta^* = \beta_{k^*}$ and $\delta = (\eps d)/(d_\mr{max} - 2 \eps d)$. By iteratively applying \cref{lem:ratio-small}, we have $|\mc{C}_{\beta^*}| \geq \min\{ \alpha n, (1+\delta)^{k^*} |\mc{C}_{\pi/2}| \}$. There are two cases. 

        \textbf{Case 1:} If $|\mc{C}_{\beta^*}| \geq \alpha n$, then notice
        \begin{equation} \label{eq:sum-up-to-k^*} \frac\pi2 - \beta^* = \beta_0 - \beta_{k^*} = \sum_{k=0}^{k^*-1} (\beta_k - \beta_{k+1}) = k^* \sin^{-1}\paren{\frac{20\alpha}{\eps}} \leq \paren{\frac{\eps}{40\alpha}} \cdot \frac\pi2 \paren{\frac{20\alpha}{\eps}} = \frac\pi4. \end{equation}
Here we have used the numerical inequality $\sin^{-1}(x) \leq \pi x/2$ for $0 \leq x \leq 1$. Therefore, $\beta^* \geq \pi/4$ and thus $|C_{\pi/4}| \geq |C_{\beta^*}| \geq \alpha n$. Hence
        \[ |\mc{C}_{\pi/4}| \sin^2(\pi/4) \geq \frac{\alpha n}{2} > \frac{5\alpha^2 n}{2}. \]
        Since $\alpha < 1/5$ this contradicts \cref{lem:contradiction}.

        \textbf{Case 2:} If $\alpha n > |\mc{C}_{\beta^*}| \geq (1+\delta)^{k^*} |\mc{C}_{\pi/2}|$, we will turn to \cref{lem:ratio-large}.  For $k > k^*$, let
        \[ \beta_k = \beta_{k-1} - \sin^{-1}\paren{\frac{d_\mr{max} + 1}{\delta} \cdot \frac{|\mc{C}_{\pi/2}|}{|\mc{C}_{\beta_{k-1}}|}}. \]

        Since $|\mc{C}_{\pi/2}|/|\mc{C}_{\beta^*}| \leq (1+\delta)^{-k^*}$, it follows by \cref{eq:condition-for-k*} that the argument of $\sin^{-1}(\cdot)$ is at most $\delta/(8(1+\delta)) < 1$, so this is well-defined.  By \cref{lem:ratio-large}, for all $k > k^*$, we have $|\mc{C}_{\beta_{k+1}}| \geq \min\{(1+\delta)|\mc{C}_{\beta_k}|, n/2\}$.  Let $M = \min\{ k : |\mc{C}_{\beta_k}| \geq n/2 \}$.  Then
        \[ \frac\pi2 - \beta_M = \beta_0 - \beta_M = \sum_{k=0}^{k^*-1} (\beta_k - \beta_{k+1}) + \sum_{k=k^*}^{M-1} (\beta_k - \beta_{k+1}). \]

        The first term is bounded by \cref{eq:sum-up-to-k^*}.  We now focus on the second term.
        \begin{align*}
            \sum_{k=k^*}^{M-1} (\beta_k - \beta_{k+1}) & \leq \sum_{k=k^*}^{M-1} \sin^{-1}\paren{\frac{d_\mr{max} + 1}{\delta} \cdot \frac{|\mc{C}_{\pi/2}|}{|\mc{C}_{\beta_k}|}}
            \leq \frac\pi2 \cdot \frac{d_\mr{max} + 1}{\delta} \sum_{k=k^*}^{M-1} \frac{|\mc{C}_{\pi/2}|}{|\mc{C}_{\beta_k}|} \\
            & \leq \frac\pi2 \cdot \frac{d_\mr{max} + 1}{\delta} \sum_{k=k^*}^{M-1} (1+\delta)^{-k}
            \leq \frac\pi2 \cdot \frac{d_\mr{max} + 1}{\delta} \sum_{k=k^*}^\infty (1+\delta)^{-k} \\
            & = \frac\pi2 \cdot \frac{d_\mr{max} + 1}{\delta} \cdot \frac{1+\delta}{\delta(1+\delta)^{k^*}}.
            = \frac\pi2 \cdot \frac{(1+\delta)(d_\mr{max} + 1)}{\delta^2 (1+\delta)^{k^*}}.
        \end{align*}

        By \cref{eq:condition-for-k*}, the total fraction is at most $\pi/16$.  Thus $\beta_M \geq 3\pi/16$, so 
        \[ |\mc{C}_{\beta_M}| \sin^2(\beta_M) \geq \frac n2 \sin^2\paren{\frac{3\pi}{16}} > \frac{n}{10}, \]
        and since $\alpha < 1/5$, we have $n/10 > 5\alpha^2n/2$, contradicting \cref{lem:contradiction}.
    \end{proof}

\section{Proof of \texorpdfstring{\cref{thm:Gnm-synchronizing}}{Theorem}}\label{sec:hitting-time}

We now proceed to the proof of \cref{thm:Gnm-synchronizing}. The crucial decomposition to apply \cref{thm:main} stems from setting $B$ to be the set of low degree vertices; near the hitting time the neighborhoods or such vertices are with high probability disjoint.  

It will be convenient to view the random graph process by coupling them through a set of uniform random variables on $[0,1]$. Let $K_n$ be an $n$-vertex complete graph on $V$.  For each $e \in E(K_n)$, consider an independent random variable $\xi_e \sim \on{Unif}([0,1])$.  For $p \in [0,1]$, let
\[ G_p = (V,E_p)\text{ and } E_p = \{ e \in E(K_n) : \xi_e \leq p \}. \]
Note that $G_p \sim G(n,p)$. 
    
Let $g(n)$ be a function tending to infinity sufficient slowly. Define the three times
    \[ \sigma = \frac{\log n - g(n)}{n-1},~\omega = \frac{5\log n}{n-1}~\text{ and }\lambda = \min\{t \in [0,1] : G_t \text{ is connected} \}. \]
    The first two are deterministic whereas the third one is random, depending on the outcomes of $\{\xi_{e}\}_{e \in E(K_n)}$. 
    Note that almost surely $\xi_e \ne \xi_{e'}$ for any $e \ne e'$.  Thus we may define the random variable $t_m = \min\{p : |E(G_p)| \geq m\}$ and have $G_{t_m} \sim G(n,m)$ almost surely. In particular, $G_\lambda \sim G(n,\tau)$ for $\tau$ the stopping time defined in \cref{thm:Gnm-synchronizing}.

Fix $\eps > 0$ sufficiently small to be chosen later.  We now partition $V$ into
\begin{equation}\label{eq:def-of-B-W}
 B = \left\{ v \in V(G) : \deg_\sigma(v) \leq 11\eps\log n \right\}\text{ and }
W = V \setminus B.
\end{equation}

The following simple lemma, which is a slight modification of \cite[Lemma~10.2]{kahn2019asymptotics} or \cite[Lemma~5.1]{devlin2017perfect} shows that the set $B$ in the above decomposition satisfies the properties required to use \cref{thm:main}. 
\begin{lemma}\label{lem:defect-edges-well-behaved}
Let $B$ and $W$ be as above. Then with high probability we have:
        \begin{enumerate}
            \item $\lambda \in (\sigma,\omega)$
            \item $|B| < n^{22\eps \log(e/(11\eps))}$
            \item No edge in $G_\omega$ has both endpoints outside $W$
            \item No vertex in $W$ has more than one neighbor outside $W$ in $G_\omega$
        \end{enumerate}
    \end{lemma}
\begin{proof}
(1) is a standard fact in the theory of random graphs, see, e.g.~\cite[Lemma~5.1(a)]{devlin2017perfect}.

For (2), observe that the distribution of the degree $\on{deg}_{G_\sigma}(u)$ is $\on{Bin}(n-1, \sigma)$. Let $\gamma = 11\varepsilon \log (e/11\varepsilon)$. and observe that any $v \in V$,
\begin{align*}
\mb{P}[v \in B] &= \mb{P}[\on{Bin}(n-1, \sigma) \le 11\eps \cdot \log n]\le \sum_{t=0}^{11\eps \cdot \log n}\binom{n}{t}\cdot \sigma^t \cdot (1-\sigma)^{n-t}\\
&\le 2(1-\sigma)^n \cdot \sum_{t=0}^{11\eps \cdot \log n}\Big(\frac{en\sigma}{t}\Big)^t\le 2n^{-1}\cdot e^{g(n)} \cdot n^{\gamma}.
\end{align*}
Therefore, $\mb{E}[|B|] < n^{3\gamma/2}$, which implies (2) by Markov's inequality.

Additionally, observe that
\begin{align*}
\mb{E}[e(B,B)] &=\binom{n}{2} \cdot \mb{P}[\on{deg}_{\sigma}(u)\le 11\eps \log n \cap \on{deg}_{\sigma}(v)\le 11\eps \log n \cap (u,v)\in G_{\omega}]\\
&\le \binom{n}{2} \cdot \omega \cdot \mb{P}[\on{deg}_{\sigma}(u)\le 11\eps \log n \cap \on{deg}_{\sigma}(v)\le 11\eps \log n]\\
&\le \omega \cdot (n \cdot \mb{P}[\on{Bin}(n-2, \sigma) \le 11\eps \cdot \log n])^2\le n^{-1/2}.
\end{align*}
Here we have used linearity of expectation, Harris's inequality, and then essentially repeated the estimate for $|B|$. The condition (3) then follows via Markov's inequality. 

Conclusion (4) follows similarly, using that
\begin{align*}
\mb{E}[\#\{ (b, b', v) : b \in B, \ b' \in B, &\ vb \in E_\omega, \ vb' \in E_\omega \} ]\\
&\le n^2 \cdot \omega^2 \cdot \mb{P}[\mb{P}[\on{deg}_{\sigma}(u)\le 11\eps \log n \cap \on{deg}_{\sigma}(v)\le 11\eps \log n] \le n^{-1/2}.
\end{align*}
The desired result then follows via Markov's inequality. \qedhere
  \end{proof}

    \begin{proposition}\label{prop:Gnp-nice}
        Fix $\eps,B,W$ as in \cref{eq:def-of-B-W}.  With high probability, the following holds simultaneously for all $p \in (\sigma,\omega)$.  The graph $G_p$ is an $(n,d,\alpha)$-expander for $d = pn$ and $\alpha = 20(\log n)^{-1/2}$.  The maximum degree of $G_p$ is at most $20 \log n$.  The minimum degree of $G_p[W]$ is at least $10(\eps + \alpha) \log n$.
    \end{proposition}

    Note that $d$ is the only parameter which changes with $p$; all others are fixed throughout the range.

    \begin{proof}
        Due to monotonicity, it suffices to check the minimum degree condition only at $G_\sigma[W]$. For all $v \in W$, by definition, $\deg_{G_\sigma}(v) \geq 11 \eps \log n$, and by \cref{lem:defect-edges-well-behaved}, $\deg_{G_\sigma[W]}(v) \geq \deg_{G_\sigma}(v) - 1$.  As $\alpha = o(1)$ and $\eps>0$ is a constant, we get
        \[ \deg_{G_\sigma[W]}(v) \geq 11\eps \log n - 1 > 10(\eps+\alpha)\log n \]
        for sufficiently large $n$.

        Similarly, by monotonicity, it suffices to check the maximum degree condition only for $G_\omega$. We will use a standard Chernoff bound for the binomial distribution: let $X_1,\dots,X_n$ be independent random variables taking values in $\{0,1\}$ and let $S_n = X_1 + \dots + X_n$. Then, for any $\delta > 0$
        \[\mb{P}[X \geq (1+\delta)\mb{E}[X]] \leq \exp(-\delta^2 \mb{E}[X]/(2+\delta))\]
        and for any $0 < \delta < 1$,
        \[\mb{P}[X \leq (1-\delta)\mb{E}[X]] \leq \exp(-\delta^2 \mb{E}[X]/2).\]
        
        For any $v \in V$, $\on{deg}_{G_\omega}(v)$ is distributed as $\on{Bin}(n-1, \omega)$, which has mean $5\log{n}$. Therefore, by the union bound, followed by the above Chernoff bound, 
        \begin{align*}
        \mb{P}[\exists v : \deg_{G_\omega} v \geq 20 \log n] \leq n \mb{P}[\deg_{G_\omega} v \geq 20 \log n] \leq n \cdot \exp(-9 \log{n}) = n^{-8}. 
        \end{align*}

        Thus we need only check that simultaneously for all $p \in (\sigma, \omega)$, $G_p$ is an $(n,d,\alpha)$-expander for $d=pn$ and $\alpha = 20(\log n)^{-1/2}$.   By \cite[Corollary 6.6]{abdalla2022expander}, we get $\mb{P}[\|A_{G_p} - d\,J_n/n\| \geq 20 \sqrt{\log n}] < 2n^{-3}$ for any fixed $p \in (\sigma,\omega)$.  By another Chernoff bound, with high probability, there are at most $10 n \log n$ edges with $\xi_e \in (\sigma,\omega)$, and so at most $10 n \log n$ distinct graphs $G_p$ for $p \in (\sigma,\omega)$.  Thus by a union bound, with high probability, all graphs $G_p$ with $p \in (\sigma,\omega)$ are $(n,d,\alpha)$-expanders for $d = pn$ and $\alpha = 20 (\log n)^{-1/2}$.
    \end{proof}

    \begin{proof}[Proof of \cref{thm:Gnm-synchronizing}]
        Suppose $p \in [\lambda,\omega)$.  Fix a sufficiently small constant $\eps > 0$ and let $B,W$ be as in \cref{eq:def-of-B-W}.  Let $d = pn$, $\alpha = 20 \log n$, $d_\mr{max} = 20 \log n$, and $\ell = 2(\eps + \alpha)d$.  By \cref{prop:Gnp-nice}, $G_p$ is a $(n,d,\alpha)$-expander; moreover, the maximum degree of $G_p[W]$ is at most $d_\mr{max}$ and minimum degree of $G_p[W]$ is at least $\ell$.  By \cref{lem:defect-edges-well-behaved}, $B$ satisfies \cref{item:B-small,item:B-empty,item:no-2-neighbors}, and by \cref{prop:Gnp-nice}, $B$ satisfies \cref{item:B-sparse}.  As $\eps = \Theta(1) = \delta$, $d_\mr{max} = \Theta(\log n) = d$, and $\alpha = \Theta((\log n)^{-1/2})$,
        \[ \frac{8(1+\delta)(d_\mr{max} + 1)}{\delta^2 \exp(\lfloor\frac{\eps}{40\alpha}\rfloor \log(1+\delta))} = \frac{\Theta(\log n)}{\exp(\Theta((\log n)^{1/2}))} = o(1), \]
        \[ \frac{8}{d_\mr{max} - 4 \eps d} = \frac{8}{\Theta(\log n)} = o(1), \]
        verifying \cref{eq:condition-for-k*,eq:condition-for-d} for sufficiently large $n$.  Since $p \geq \lambda$, there are no isolated vertices.  Thus by \cref{thm:main}, $G_p$ is globally synchronizing simultaneously for all $p \in [\lambda,\omega)$ with high probability.

        The regime $p \geq \omega$ is easier and follows essentially from \cite[Proposition~6.9]{abdalla2022expander}; we sketch the necessary modifications for completeness. Let $d = pn$ as before. By \cite[Corollary 6.6]{abdalla2022expander}, $\alpha \leq 10 d^{-1/2}$ with probability at least $1 - 2n^{-3}$ for $p \geq \omega$.  By a Chernoff bound as above, for $p \geq \omega$, we have $\mb{P}[d_\mr{max}(G_p) > 20d] < n^{-8}$ and $\mb{P}[d_\mr{min}(G_p) < 0.1 d] < \exp(-(0.9)^2\cdot 5\log n/2) = n^{-2.025}$.  Thus by \cite[Proposition 3.2 and Theorem 1.10]{abdalla2022expander}, for any $p \geq \omega$, $G_p$ is globally synchronizing with probability $1 - 2n^{-2.025}$.  Since there are at most $\binom n2$ distinct graphs in the evolution of $G_p$, by a union bound, with probability $1-n^{-0.025}$, all graphs $G_p$ for $p \geq \omega$ are globally synchronizing. \qedhere
        \end{proof}




    \bibliographystyle{amsplain0.bst}
    \bibliography{main}

\providecommand{\bysame}{\leavevmode\hbox to3em{\hrulefill}\thinspace}
\providecommand{\MR}{\relax\ifhmode\unskip\space\fi MR }
\providecommand{\MRhref}[2]{%
  \href{http://www.ams.org/mathscinet-getitem?mr=#1}{#2}
}
\providecommand{\href}[2]{#2}
\begin{thebibliography}{1}

\bibitem{abdalla2022expander}
Pedro Abdalla, Afonso~S Bandeira, Martin Kassabov, Victor Souza, Steven~H Strogatz, and Alex Townsend, \emph{Expander graphs are globally synchronising}, arXiv:2210.12788.

\bibitem{devlin2017perfect}
Pat Devlin and Jeff Kahn, \emph{Perfect fractional matchings in {$k$}-out hypergraphs}, Electron. J. Combin. \textbf{24} (2017), Paper No. 3.60, 12.

\bibitem{kahn2019asymptotics}
Jeff Kahn, \emph{Asymptotics for {S}hamir's problem}, Adv. Math. \textbf{422} (2023), Paper No. 109019, 39.

\bibitem{kassabov2021sufficiently}
Martin Kassabov, Steven~H. Strogatz, and Alex Townsend, \emph{Sufficiently dense {K}uramoto networks are globally synchronizing}, Chaos \textbf{31} (2021), Paper No. 073135, 7.

\bibitem{kuramoto1975self}
Yoshiki Kuramoto, \emph{Self-entrainment of a population of coupled non-linear oscillators}, International {S}ymposium on {M}athematical {P}roblems in {T}heoretical {P}hysics ({K}yoto {U}niv., {K}yoto, 1975), Lecture Notes in Phys., vol.~39, Springer, Berlin-New York, 1975, pp.~420--422.

\bibitem{ling2019landscape}
Shuyang Ling, Ruitu Xu, and Afonso~S. Bandeira, \emph{On the landscape of synchronization networks: a perspective from nonconvex optimization}, SIAM J. Optim. \textbf{29} (2019), 1879--1907.

\end{thebibliography}



\end{document}